\documentclass{article}
\usepackage{amssymb}
\usepackage{amsfonts}
\usepackage{amsmath}
\usepackage{amsthm}
\usepackage{amsfonts}
\usepackage[all]{xy}

\usepackage{enumerate}
\usepackage{setspace}
\usepackage{comment}
\usepackage{graphicx}
\usepackage{float}
\usepackage{tikz}
\usepackage{cite}
\tikzstyle{vertex}=[circle, draw, inner sep=0pt, minimum size=6pt]

\usepackage{calc}
\usepackage{pgfplots}
\usepackage{subfigure}
\usepackage[percent]{overpic}

\usepackage[english]{babel}
\usepackage{graphicx}
\usepackage{amsmath}
\usepackage{amsthm}
\usepackage{amsfonts}
\usepackage{amssymb}
\usepackage{mathabx}
\usepackage{epstopdf}
\usepackage{xcolor}
\usepackage{enumerate}
\usepackage{comment}
\usepackage{bm}
\usepackage{cite}
\theoremstyle{plain}
\newtheorem{theorem}{Theorem}[section]
\newtheorem{lem}[theorem]{Lemma}

\newtheorem{prop}[theorem]{Proposition}

\usepackage[colorlinks=true,citecolor=black,linkcolor=black,urlcolor=blue]{hyperref}
\usepackage {tikz}
\usetikzlibrary{automata,arrows,calc,positioning}
\usetikzlibrary{ decorations.pathreplacing}
\usetikzlibrary{positioning, decorations.text}
\usepackage{pdfpages}
\usepackage{pgf,tikz,pgfplots}
\pgfplotsset{compat=1.14}
\usepackage{mathrsfs}
\usepackage{mathtools}
\usetikzlibrary{arrows}

\newcommand{\RelKp}[2]{\mbox{Rel}_{#2}(#1;p)}

\newcommand{\splitRelp}[3]{\mbox{splitRel}_{#2,#3}(#1;p)}

\newcommand{\qqed}{\tag*{$\qedsymbol$}}

\begin{document}
\definecolor{ududff}{rgb}{0.30196078431372547,0.30196078431372547,1}

\title{On the Split Reliability of Graphs}
\author{Jason I. Brown
and 
Isaac McMullin  \\ Department of Mathematics and Statistics, Dalhousie University\\ \\}

\maketitle

%\noindent
%\textit{Corresponding Author:} \\
%Jason Brown \\
%Department of Mathematics and Statistics, Dalhousie University\\
%6316 Coburg Road, PO BOX 15000\\
%Halifax, Nova Scotia, Canada B3H 4R2\\
%\textit{Email:} jason.brown@dal.ca \\

%\clearpage

\begin{abstract}
A common model of robustness of a graph against random failures has all vertices operational, but the edges independently operational with probability $p$. One can ask for the probability that all vertices can communicate ({\em all-terminal reliability}) or that two specific vertices (or {\em terminals}) can communicate with each other ({\em two-terminal reliability}). 
%While both of these questions have been well-studied, they are both increasing functions of the edge probability. 
A relatively new measure is {\em split reliability}, where for two fixed vertices $s$ and $t$, we consider the probability that every vertex communicates with one of $s$ or $t$, but not both. In this paper, we explore the existence for fixed numbers $n \geq 2$ and $m \geq n-1$ of an {\em optimal} connected $(n,m)$-graph $G_{n,m}$ for split reliability, that is, a connected graph with $n$ vertices and $m$ edges for which for any other such graph $H$, the split reliability of $G_{n,m}$ is at least as large as that of $H$, for {\em all} values of $p \in [0,1]$. Unlike the similar problems for all-terminal and two-terminal reliability, where only partial results are known, we completely solve the issue for split reliability, where we show that there is an optimal $(n,m)$-graph for split reliability if and only if $n\leq 3$, $m=n-1$, or $n=m=4$.

%\iffalse
%In this paper, we explore the idea of the existence of optimal connected $(n,m)$-graphs for split reliability. In other words, if we are given a number of vertices ($n$) and edges ($m$), does there exist a connected $(n,m)$-graph whose split reliability is greater than or equal to that of every other connected $(n,m)$-graph when $p\in[0,1]$? To examine this problem, we first find a way to easily compare the split reliability of two different graphs as we take $p$ sufficiently close to $0$ or $1$. Using this, we show that if an optimal graph exists for given values of $n$ and $m$, it must be of a very specific type, so if we can find a graph whose split reliability is greater than that type's split reliability anywhere on $p\in[0,1]$, no optimal graph exists in that case. Using this proposition, we find that if $n\geq 2$ and $m\geq n-1$, there is an optimal $(n,m)$-graph if and only if $n\leq 3$, $m=n-1$, or $n=m=4$.
%\fi
\end{abstract}

\vspace{0.25in}
\noindent \textit{Proposed running head}: Split Reliability of Graphs 

%\maketitle
%%\section{}
%%\subsection{}
%%{Abstract}
%
%\begin{abstract}
%.\\
 %
%\noindent \textit{Keywords}: graph, polynomial, zero, Beraha-Kahane-Weiss Theorem
%\end{abstract}

\section{Introduction} %===================================
A graph (or {\em network}) $G=(V,E)$ consists of a finite vertex set $V$ and a finite edge multiset $E$ consisting of unordered pairs of $V$. The multiset of edges of $G$ that have the same endpoints as edge $e$ of $G$ is denoted by $[e]$, and $|[e]|$ is called the {\em multiplicity} of $[e]$. 
%If $H=(V',E')$ is a graph with $V'\subseteq V$ and $E'\subseteq E$ where $E'$ consists of unordered pairs and singletons from $V'$, then $H$ is called a {\em subgraph} of $V$. 
The {\em order} and {\em size} of a graph $G$ are $|V|$ and $|E|$, respectively. If $G$ is a graph of order $n$ and size $m$, we refer to $G$ as an ($n,m$)-graph. For standard graph theory terminology, we refer the reader to \cite{west}.

There are many different ways we can model the robustness of a network to random failures, but in the most common models, you start with a graph $G$, and analyze the probability of the graph being in a specific ``operational'' state, given that while all vertices are always operational, each edge independently has probability $p$ of being operational \cite{colbook}. For example, the {\em all-terminal reliability} problem is to find the probability that all vertices in $G$ can communicate, and in the {\em two-terminal reliability} problem we would fix two vertices $s$ and $t$ and ask for the probability that $s$ and $t$ can communicate. More generally, let $K$ be a non-empty subset of vertices. The $K$-terminal reliability of $G$ is given by
\begin{eqnarray}
\RelKp{G}{K} & = & \sum_{E^\prime} p^{|E^\prime|}(1-p)^{|E - E^\prime|} \label{kTerminalReliability},
\end{eqnarray}
where the sum is over all subsets $E^\prime$ of edges of $G$ that connect all vertices of $K$. The all-terminal and two-terminal reliability problems correspond to $K=V$ and $K=2$ respectively. It is known that the problem of finding the $K$-terminal reliability of a graph $G$ is $\#P$-complete, even in the special cases of all-terminal and two-terminal reliabilities (for details, see \cite{sharpp}).

For example, if $T_{n}$ is a tree of order $n$, then clearly $\displaystyle{\RelKp{T_{n}}{V} = p^{n-1}}$
%since every edge needs to be active in order for the graph to be connected. If $s$ and $t$ are two specified vertices of $T_{n}$ and $k$ is the length of the path between $s$ and $t$ (only one of these paths exists since $T_{n}$ is a tree and thus contains no cycles), then
and if $s$ and $t$ are two specified vertices of $T_{n}$ and $k$ is the length of (i.e. the number of edges in) the unique path between $s$ and $t$, then
$\displaystyle{\RelKp{T_n}{\{s,t\}} = p^{k}}$.
%since we would need every edge on that path to be active in order for $s$ and $t$ to be able to communicate.
If $C_{n}$ is a cycle of order $n$, then
$\displaystyle{\RelKp{C_{n}}{V} = p^{n}+np^{n-1}(1-p)},$
%since we either have all edges active, or any choice of one edge inactive which leaves us with a connected path of length $n-1$. Removing any more edges would result in a disconnected graph. If $s$ and $t$ are two specified vertices of $C_{n}$ and $k$ is the length of the shortest path between $s$ and $t$ (there are exactly two paths connecting $s$ and $t$, one of length $k$ and one of length $n-k$), then
and if $s$ and $t$ are two specified vertices of $C_{n}$ with $k$ as the length of the shortest path between $s$ and $t$, then
$\displaystyle{\RelKp{C_n}{\{s,t\}} = p^{k}+p^{n-k}-p^{n}.}$
%since in order for $s$ and $t$ to be able to communicate, we need at least one of the two paths between $s$ and $t$ to be active.
Additionally, consider the graph $G_{1}$ in Figure \ref{figureG1}.
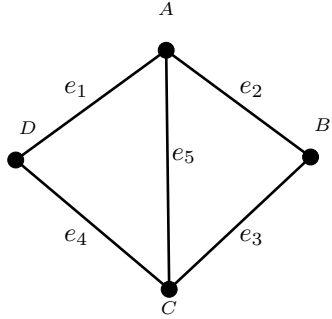
\begin{figure}
\centering
\begin{tikzpicture}[line cap=round,line join=round,>=triangle 45,x=1cm,y=1cm]
\clip(-17,2) rectangle (15.21,7);
\draw [line width=1pt] (-8.85,6.29)-- (-10.85,4.83);
\draw [line width=1pt] (-10.85,4.83)-- (-8.81,3.11);
\draw [line width=1pt] (-8.81,3.11)-- (-6.93,4.87);
\draw [line width=1pt] (-8.85,6.29)-- (-6.93,4.87);
\draw [line width=1pt] (-8.85,6.29)-- (-8.81,3.11);
\draw (-10.33,6) node[anchor=north west] {\textit{$e_{1}$}};
\draw (-8,6) node[anchor=north west] {\textit{$e_{2}$}};
\draw (-8,4.03) node[anchor=north west] {\textit{$e_{3}$}};
\draw (-10.33,4.03) node[anchor=north west] {\textit{$e_{4}$}};
\draw (-8.9,5.11) node[anchor=north west] {\textit{$e_{5}$}};
\begin{scriptsize}
\draw [fill=black] (-8.85,6.29) circle (3pt);
\draw[color=black] (-8.85,6.84) node {$A$};
\draw [fill=black] (-6.93,4.87) circle (3pt);
\draw[color=black] (-6.77,5.3) node {$B$};
\draw [fill=black] (-8.81,3.11) circle (3pt);
\draw[color=black] (-8.81,2.86) node {$C$};
\draw [fill=black] (-10.85,4.83) circle (3pt);
\draw[color=black] (-10.69,5.26) node {$D$};
\end{scriptsize}
\end{tikzpicture}
\caption{Graph $G_1$}
\label{figureG1}
\end{figure}
%To calculate the all-terminal reliability of $G_{1}$, we can count the operational states just as we did in (\ref{kTerminalReliability}). There is one operational state with $5$ edges, five with $4$ edges (removing any single edge still leaves a connected graph), eight with $3$ edges, and none with fewer than $3$ edges (as you need at least a spanning tree up in order to have a connected subgraph). Thus we have
Then
$\RelKp{G_1}{V} = 4p^{5}-11p^{4}+8p^{3}$
%Suppose we now want to calculate the two-terminal reliability of $G_{1}$ with $D$ as $s$ and $B$ as $t$. There are two operational states with $2$ edges, eight with $3$ edges, five with $4$ edges, and one with all edges active. Any fewer edges active and $D$ and $B$ cannot communicate since the distance between $D$ and $B$ is $2$. Thus we have
while $\RelKp{G_1}{\{D,B\}} = 2p^{5} - 5p^{4} + 2p^{3} + 2p^{2}.$
%as the two-terminal reliability of this graph in terms of $B$ and $D$.

%With all-terminal and two-terminal reliabilities established, we now move on to a third form of the reliability problem on a network.

A new form of reliability, {\em split reliability}, was proposed in \cite{brownmol}, in the study of the all-terminal reliability of ``gadget replacements'' of edges in a graph (see \cite{brownmol} for details). In split reliability, we fix two vertices (or {\em terminals}) $s$ and $t$ in our graph $G$. Edges are again independently operational with probability $p$, but the split reliability of $G$ with terminals $s$ and $t$ is the probability that every vertex in the graph can communicate with either $s$ or $t$ {\em but not both} (that is, the probability that the graph will be split into exactly two components, one containing $s$ and the other containing $t$). We will use $\splitRelp{G}{s}{t}$ to represent the split reliability polynomial of $G$ with a specific choice of $s$ and $t$. Our interest in split reliability arises not only from the connection to all-terminal reliability, but as a model of robustness of a network where it may be important that information is transmitted from two terminals to all other vertices, but without the possibility that a vertex gets possibly conflicting messages.

%For example of calculating split reliability, consider the graph $G_{1}$ in Figure \ref{figureG1}. Then
%\begin{eqnarray*}
%\splitRelp{G_{1}}{A}{C} &=& -4p^{5} + 12p^{4} - 12p^{3} + 4p^{2}.
%\end{eqnarray*}
%While
%\begin{eqnarray*}
%\splitRelp{G_{1}}{B}{D} &=& -6p^{5} + 20p^{4} - 22p^{3} + 8p^{2}.
%\end{eqnarray*}

It is useful to have a way to calculate the split reliability of a graph by counting {\em operational states} (that is, spanning subgraphs consisting of two components, one containing $s$ and the other containing $t$), along the lines of (\ref{kTerminalReliability}).
%Suppose that in an $(n,m)$-graph $G$ with fixed vertices $s$ and $t$, there are $N_{i}$ operational states with exactly $i$ edges active. We have that the probability that the graph is split into exactly two components with one containing $s$ and the other containing $t$ with $i$ edges active is $N_{i}p^{i}(1-p)^{m-i}$. With this and $(\ref{kTerminalReliability})$ in mind, we have the following theorem.
\begin{theorem}\label{Nform}
For any graph $G$ of order $n\geq 2$, size $m$, and any distinct vertices $s$ and $t$ of $G$,
\begin{eqnarray}
\splitRelp{G}{s}{t} & = & \sum_{i=n-2}^{m-c}N_{i}p^{i}(1-p)^{m-i},\label{splitReliabilityN}
\end{eqnarray}
where $N_{i}$ is the number of operational states for split reliability (with the specified $s$ and $t$) with exactly $i$ edges and  $c$ is the minimum cardinality of an $(s,t)$-cutset, i.e. $c$ is the smallest number of edges you need to remove to disconnect $s$ from $t$. 
\end{theorem}

\noindent Note that if $G$ is connected, $c\geq 1$ so we can always run the summation from $n-2$ to $m-1$ instead (of course, $N_{i}=0$ for $i\geq m-c+1$).

Let's examine some examples of split reliability. For a tree $T_{n}$ of order $n$, with terminals $s$ and $t$ at distance $k$, 
\begin{eqnarray*}
\splitRelp{T_{n}}{s}{t}& = &kp^{n-2}(1-p).
\end{eqnarray*}
%This is because we need to have exactly one edge down on the path of length $k$ to ensure that $s$ and $t$ cannot communicate, and every other edge must be up in order for every other vertex to communicate with exactly one of $s$ or $t$.
For the cycle $C_{n}$ of order $n$, with terminals $s$ and $t$ at distance $k$, 
\begin{eqnarray*}
\splitRelp{C_{n}}{s}{t}& = &k(n-k)p^{n-2}(1-p)^{2}.
\end{eqnarray*}
%because we must have one edge down on each path in order for $s$ and $t$ to be in separate components, and every other edge must be up in order for every other vertex to communicate with exactly one of $s$ or $t$. Thus, we have $k(n-k)$ choices for the two vertices to take down.
Again, consider the graph $G_{1}$ in Figure \ref{figureG1}, with $D$ and $B$ as the terminals. It can easily be seen that $c=2$ so we only need to find $N_{2}$ and $N_{3}$ since $G_{1}$ is a $(4,5)$-graph. By counting operational states, we see that $N_{2}=8$ and $N_{3}=2$. Thus we have
\begin{eqnarray*}
\splitRelp{G_{1}}{D}{B} &=& 8p^{2}(1-p)^{3}+2p^{3}(1-p)^{2}\\
&=& -6p^{5}+20p^{4}-22p^{3}+8p^{2}
\end{eqnarray*}
Suppose we also want to find the split reliability of $G$ with $A$ and $C$ as the terminals. We have that $c=3$, so all we have to do is find $N_{2}$. We find that $N_{2}=4$ and so we have
\begin{eqnarray*}
\splitRelp{G_1}{A}{C}& = &4p^{2}(1-p)^{3}\\
&=&-4p^{5} + 12p^{4} - 12p^{3} + 4p^{2}.
\end{eqnarray*}
Thus, the split reliability of a graph is dependent on the specific choices of $s$ and $t$. Figure~\ref{plots} illustrates a fundamental difference between split reliability and all-terminal reliability -- namely that while all-terminal reliabilities of connected graphs of order at least 2 are increasing functions with values of $0$ and $1$ at $p = 0$ and $p = 1$, respectively, split reliabilities of connected graphs of order at least 3 are $0$ at {\em both} $p = 0 $ and $p = 1$ (and hence are \underline{not} increasing functions of $p \in [0,1]$).

\begin{figure}
    \centering
    \includegraphics[width=3.0in]{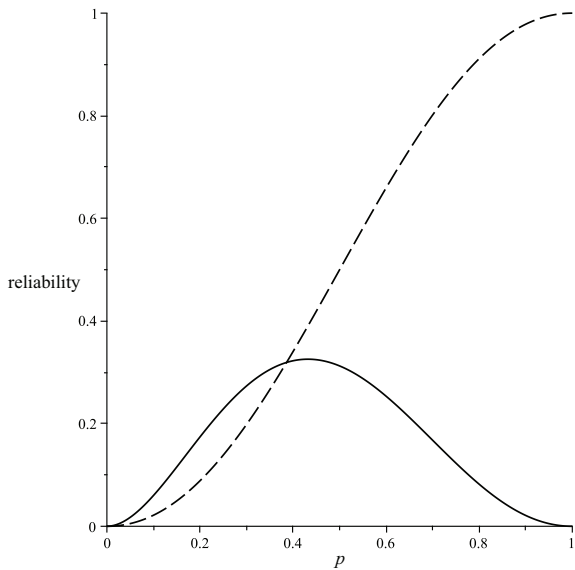}
    \caption{ The split reliability of $G_1$ with terminals $D$ and $B$ (solid), along with the all-terminal reliability of $G_1$ (dashed).}
    \label{plots}
\end{figure}

If we are given a (connected) graph $G$ with distinct vertices $s$ and $u$, form a new graph $G^\prime$ by adjoining a new vertex $t$ to $u$. Then it is easy to see that 
\[ \splitRelp{G^\prime}{s}{t} = (1-p) \cdot \RelKp{G}{V} + p \cdot \splitRelp{G}{s}{u},\]
so the $\#P$-completeness of split reliability follows from that of all-terminal reliability. 

We note that if the graph $G$ is disconnected, the split reliability of the graph is $0$ unless the graph has exactly two components with $s$ in one component and $t$ in the other. If the graph has exactly two components $A$ and $B$, with $s$ in one component and $t$ in the other, then we can see that the split reliability of the graph is the product of the all-terminal reliabilities of $A$ and $B$. For this reason, we restrict our study of split reliability to connected graphs, as those are the interesting new cases that do not reduce to all-terminal reliability.

With this in mind, we will explore the question of when there is an optimal graph for split reliability: if we are given a specific number of vertices and edges, can we create a connected graph that has split reliability greater than or equal to all other such graphs of the same order and size, no matter what the value of $p\in[0,1]$ is?

\section{Optimal Graphs}
A connected $(n,m)$-graph $G$ with terminals $s$ and $t$ such that for any other connected $(n,m)$-graph $H$, any terminals $s^\prime$ and $t^\prime$, and {\em any} $p\in (0,1)$, $\splitRelp{G}{s}{t} \geq \splitRelp{H}{s^\prime}{t^\prime}$ is called an {\bf optimal} $\bm{(n,m)}${\bf -graph}, or simply an {\bf optimal graph}, if $n$ and $m$ are understood or fixed. (As previously noted, we restrict ourselves to connected graphs.) 

An optimal graph is one in which, given fixed resources in terms of the number of vertices and edges, all that determines the best network is the underlying structure, not the specific edge probability.
The study of similar notions of optimal graphs has occupied a number of researchers in other areas of reliability (such as all-terminal and two-terminal), where only partial results are known about the existence or non-existence of such graphs. The reader can consult a recent survey of uniformly optimally reliable graphs in \cite{romero} as well as practical guidelines in \cite{brownhighway}. A study of most-reliable multigraphs can be found in \cite{browncoxoptimal,gross,martinez}.

Surprisingly, split reliability turns out to be the only form of reliability for which we can determine precisely for what values of $n$ and $m$ such optimal graphs exist. It is to that end that we devote the remainder of this paper.

\vspace{0.25in}

Clearly, there is only one $(2,m)$-graph for all $m$ (namely a bundle of $m$ edges between the two vertices), so trivially there is always an optimal $(2,m)$-graph. We therefore assume $n \geq 3$ for the rest of this section.
As we are restricting to connected graphs, if a graph has $n$ vertices and $m$ edges, clearly $m \geq n-1$, with equality if and only if $G$ is a tree. We begin with the base case, $m = n-1$.

\begin{prop}\label{proptrees}
For all $n \geq 3$, there is an optimal $(n,n-1)$-graph.
\end{prop}
\begin{proof}
Since all of the graphs we are considering are connected, any $(n,n-1)$-graph $G$ is a tree. As noted previously, if $s$ and $t$ are distinct vertices of such a graph $G$, the split reliability of $G$ is
\begin{eqnarray*}
\splitRelp{G}{s}{t}=kp^{n-2}(1-p),
\end{eqnarray*}
where $k$ is the length of the unique path between $s$ and $t$ (with $1 \leq k \leq n-1$). In order to optimize the split reliability of $G$, all we need to do is maximize the value of $k$. Thus, for split reliability, the optimal $(n,n-1)$-graph is the path of order $n$, with $s$ and $t$ as the endpoints of the path. 
\end{proof}

We now only need to consider when $m \geq n$. Before we begin, we shall need the following useful observation.

\begin{lem}
\label{mainobservation}
Suppose $n \geq 3$ and that $G_1$ and $G_2$ are two $(n,m)$-graphs, with $s_i$ and $t_i$ being two distinct vertices of $G_i$ ($i = 1,2$). Suppose further that $N_{j,i}$ is the number of spanning subgraphs of $G_j$ with $i$ edges that consists of two components, one containing $s_j$, the other $t_j$ ($j=1,2$), so that
\begin{eqnarray}\label{sprel1}
\splitRelp{G_1}{s_1}{t_1} & = & \sum_{i=n-2}^{m-1}N_{1,i}p^{i}(1-p)^{m-i}
\end{eqnarray}
and
\begin{eqnarray}\label{sprel2}
\splitRelp{G_2}{s_2}{t_2} & = & \sum_{i=n-2}^{m-1}N_{2,i}p^{i}(1-p)^{m-i}.
\end{eqnarray}
Then if $N_{2,m-1}>N_{1,m-1}$, then for all $p$ sufficiently close to $1$, 
\[ \splitRelp{G_2}{s_2}{t_2} > \splitRelp{G_1}{s_1}{t_1},\]
and if $N_{2,n-2}>N_{1,n-2}$, then for all $p$ sufficiently close to $0$, 
\[ \splitRelp{G_2}{s_2}{t_2} > \splitRelp{G_1}{s_1}{t_1}.\qqed\] 
\end{lem}

\begin{prop}\label{onlypossibleoptimal}
For $m \geq n \geq 3$, if there exists an optimal $(n,m)$-graph, it must be of the form $G = G_{n,m}$ that consists of a single path of length $n-1$ between vertices $s$ and $t$ and all extra edges bundled between two specific adjacent vertices on the path (see Figure~\ref{bestpossible}).
\end{prop}
\begin{proof}
We begin by remarking that the specific choice of adjacent vertices of $G = G_{n,m}$ to carry the bundle of $m-n+2$ edges is irrelevant, as the split reliability in any case is clearly \[(n-2)(1-p)p^{n-3}(1-(1-p)^{m-n+2})+(1-p)^{m-n+2}p^{n-2}.\]
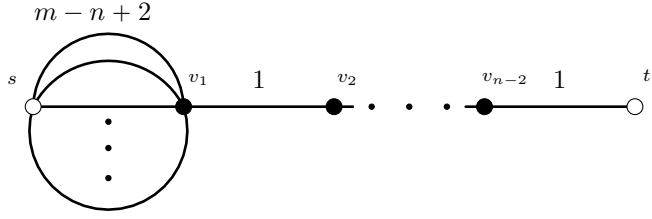
\begin{figure}
    \centering
\begin{tikzpicture}[line cap=round,line join=round,>=triangle 45,x=1cm,y=1cm]
\clip(-15,3.9) rectangle (-6,7);
\draw [shift={(-13.39,5.269536082474226)},line width=1pt]  plot[domain=0.030454498754235063:3.111138154835558,variable=\t]({1*1.0004639175257715*cos(\t r)+0*1.0004639175257715*sin(\t r)},{0*1.0004639175257715*cos(\t r)+1*1.0004639175257715*sin(\t r)});
\draw [shift={(-13.39,4.78532786885246)},line width=1pt]  plot[domain=0.4753162993630909:2.6662763542267025,variable=\t]({1*1.1246721311475405*cos(\t r)+0*1.1246721311475405*sin(\t r)},{0*1.1246721311475405*cos(\t r)+1*1.1246721311475405*sin(\t r)});
\draw [line width=1pt] (-14.39,5.3)-- (-12.39,5.3);
\draw [shift={(-13.39,4.979963503649636)},line width=1pt]  plot[domain=-3.4513287040499847:0.30973605046019176,variable=\t]({1*1.0499635036496346*cos(\t r)+0*1.0499635036496346*sin(\t r)},{0*1.0499635036496346*cos(\t r)+1*1.0499635036496346*sin(\t r)});
\draw [line width=1pt] (-12.39,5.3)-- (-10.39,5.3);
\draw [line width=1pt] (-8.39,5.3)-- (-6.39,5.3);
\draw (-14.5,6.8) node[anchor=north west] {$m-n+2$};
\draw (-11.6,5.9) node[anchor=north west] {1};
\draw (-7.6,5.9) node[anchor=north west] {1};
\draw [line width=1pt] (-10.39,5.3)-- (-10.14,5.3);
\draw [line width=1pt] (-8.64,5.3)-- (-8.39,5.3);
\begin{scriptsize}
\draw [fill=white] (-14.39,5.3) circle (3pt);
\draw[color=black] (-14.67,5.66) node {$s$};
\draw [fill=black] (-12.39,5.3) circle (3pt);
\draw[color=black] (-12.19,5.66) node {$v_{1}$};
\draw [fill=black] (-10.39,5.3) circle (3pt);
\draw[color=black] (-10.21,5.66) node {$v_{2}$};
\draw [fill=black] (-13.39,5.1) circle (1pt);
\draw [fill=black] (-13.39,4.75) circle (1pt);
\draw [fill=black] (-13.39,4.35) circle (1pt);
\draw [fill=black] (-9.89,5.3) circle (1pt);
\draw [fill=black] (-9.39,5.3) circle (1pt);
\draw [fill=black] (-8.89,5.3) circle (1pt);
\draw [fill=black] (-8.39,5.3) circle (3pt);
\draw[color=black] (-8.12,5.66) node {$v_{n-2}$};
\draw [fill=white] (-6.39,5.3) circle (3pt);
\draw[color=black] (-6.23,5.72) node {$t$};
\end{scriptsize}
\end{tikzpicture}
    \caption{Example of $G_{n,m}$, the only candidate optimal $(n,m)$-graph (up to placement of the bundle of size $m-n+2$).}
    \label{bestpossible}
\end{figure}

We know that the split reliability of $G$ can be calculated by considering the operational states for split reliability:
\begin{eqnarray*}
\splitRelp{G}{s}{t} & = & \sum_{i=n-2}^{m-c}N^{G}_{i}p^{i}(1-p)^{m-i},
\end{eqnarray*}
where $N^{G}_{i}$ is the number of operational states with $i$ edges operational and $c$ is the minimum cardinality of an $(s,t)$-cutset. With $G$, we can see that $c=1$, and that $N^{G}_{m-1} = n-2$. Moreover, $N^{G}_{n-2} = 1 + (n-2)(m-(n-2))$, as the spanning subgraphs of $G$ of size $n-2$ with two components, one containing $s$, the other $t$, are formed by either removing the bundle of edges, or by keeping one of the edges in the bundle of cardinality $m-(n-2)$ and deleting one of the $n-2$ other edges, independently.

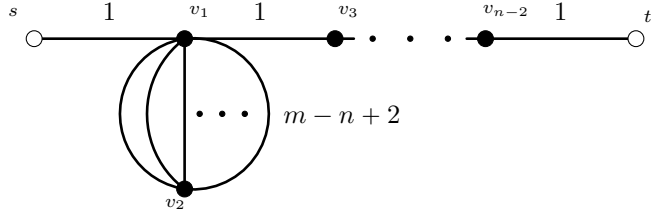
\begin{figure}
    \centering
\begin{tikzpicture}[line cap=round,line join=round,>=triangle 45,x=1cm,y=1cm]
\clip(-15,2.9) rectangle (-6,7);
\draw [line width=1pt] (-12.39,5.3)-- (-10.39,5.3);
\draw [line width=1pt] (-8.39,5.3)-- (-6.39,5.3);
\draw (-11.6,5.9) node[anchor=north west] {1};
\draw (-7.6,5.9) node[anchor=north west] {1};
\draw [line width=1pt] (-14.39,5.3)-- (-12.39,5.3);
\draw [shift={(-12.23860465116279,4.3)},line width=1pt]  plot[domain=1.721050638616422:4.562134668563164,variable=\t]({1*1.0113953488372103*cos(\t r)+0*1.0113953488372103*sin(\t r)},{0*1.0113953488372103*cos(\t r)+1*1.0113953488372103*sin(\t r)});
\draw [shift={(-11.64,4.3)},line width=1pt]  plot[domain=2.2142974355881813:4.068887871591405,variable=\t]({1*1.25*cos(\t r)+0*1.25*sin(\t r)},{0*1.25*cos(\t r)+1*1.25*sin(\t r)});
\draw [shift={(-12.276428571428571,4.3)},line width=1pt]  plot[domain=-1.683883200684532:1.6838832006845321,variable=\t]({1*1.0064285714285715*cos(\t r)+0*1.0064285714285715*sin(\t r)},{0*1.0064285714285715*cos(\t r)+1*1.0064285714285715*sin(\t r)});
\draw [line width=1pt] (-12.39,5.3)-- (-12.39,3.3);
\draw (-13.6,5.9) node[anchor=north west] {1};
\draw (-11.2,4.53) node[anchor=north west] {$m-n+2$};
\draw [line width=1pt] (-10.39,5.3)-- (-10.14,5.3);
\draw [line width=1pt] (-8.64,5.3)-- (-8.39,5.3);
\begin{scriptsize}
\draw [fill=white] (-14.39,5.3) circle (3pt);
\draw[color=black] (-14.67,5.66) node {$s$};
\draw [fill=black] (-12.39,5.3) circle (3pt);
\draw[color=black] (-12.19,5.66) node {$v_{1}$};
\draw [fill=black] (-10.39,5.3) circle (3pt);
\draw[color=black] (-10.21,5.66) node {$v_{3}$};
\draw [fill=black] (-9.89,5.3) circle (1pt);
\draw [fill=black] (-9.39,5.3) circle (1pt);
\draw [fill=black] (-8.89,5.3) circle (1pt);
\draw [fill=black] (-8.39,5.3) circle (3pt);
\draw[color=black] (-8.12,5.66) node {$v_{n-2}$};
\draw [fill=white] (-6.39,5.3) circle (3pt);
\draw[color=black] (-6.23,5.6) node {$t$};
\draw [fill=black] (-12.39,3.3) circle (3pt);
\draw[color=black] (-12.51,3.1) node {$v_{2}$};
\draw [fill=black] (-12.19,4.3) circle (1pt);
\draw [fill=black] (-11.89,4.3) circle (1pt);
\draw [fill=black] (-11.59,4.3) circle (1pt);
\end{scriptsize}
\end{tikzpicture}
    \caption{A path of order $n-1$ with an adjacent vertex incident with all remaining edges.}
    \label{fig:H}
\end{figure}

Let $H$ be an $(n,m)$-graph with terminals $u$ and $v$ that does not have the same form as $G$ (a path with one edge bundled with the remaining edges, with terminals at the ends of the path). By Lemma~\ref{mainobservation}, it suffices to show that the number $N^{H}_{m-1}<n-2$ or $N^{H}_{n-2} < 1+(n-2)(m-(n-2))$. Since $N^{H}_{m-1}$ is the number of states with $m-1$ edges operational and two components with one containing $u$ and the other containing $v$, $N^{H}_{m-1}$ is the number of single edges whose removal disconnects $u$ and $v$, that is, the number of $(u,v)$-cutsets of cardinality $1$ in $H$.

Note that $N^{G}_{m-1} = n-2 > 0$, so if $N^{H}_{m-1} = 0$, we are done. Thus we can assume that $N^{H}_{m-1} \geq 1$, that is, it only takes the removal of one edge to disconnect $u$ and $v$. Let $P$ be a shortest $u$-$v$ path in $H$, and denote its length by $k$. Every edge that is not on this path is not a $(u,v)$-cutset, so the $(u,v)$-cutsets of cardinality $1$ are some individual edges of $P$ (and not any edge of $P$ where there is another edge of $H$ with the same endpoints) and so $N^{H}_{m-1} \leq k$.  
%If $k=1$, then $u$ and $v$ are adjacent with a single edge between them, and since the removal of any other edge leaves $s$ and $t$ connected, it follows that  $N^{H}_{m-1} = n-2 = 1$, so $n = 3$. In this case, as $H$ is connected and $uv$ is a $(u,v)$-cutset, the other vertex must be joined by a bundle of $m-1$ edges to either $u$ or $v$ (but not both, so $H$ is of the same form as $G = G_{n,m}$ in this case. THus we can assume that $k \geq 2$.
As $P$ is a shortest $u$-$v$ path in $H$ and at least one edge of $P$ must be a $(u,v)$-cutset, no edge outside of $P$ joins two nonadjacent vertices of $P$, and there is no vertex $x$ not on $P$ that is joined to two nonadjacent vertices of $P$ (if the latter occurs, then at least two edges of $P$ are not $(u,v)$-cutsets, so in this case, $N^{H}_{m-1} \leq k-2 < n-2$). If $k \leq n-3$, then clearly $N^{H}_{m-1} \leq k < n-2$, and we are done. Otherwise, either $k = n-2$ or $k = n-1$. 

If $k = n-2$, then $N^{H}_{m-1} < n-2$ unless $H$ consists of the path $P$ and another vertex $w$ off the path that is joined by a bundle of $m-(n-2)$ edges to a single vertex of the path  (see Figure~\ref{fig:H}). While here $N^{H}_{m-1} = n-2$, we find that $N^{H}_{n-2} = (n-2)(m-(n-2))$, as the spanning subgraphs of $H$ with two components, one containing $u$, the other $v$, are formed by keeping one of the edges in the bundle of cardinality $m-(n-2)$ and deleting one of the $n-2$ other edges, independently. However, then $N^{H}_{n-2} < 1+(n-2)(m-(n-2)) = N^{G}_{n-2}$, and we are done.

Finally, if $k = n-1$, then $H$ consists of a path between $u$ and $v$, with some of the edges in the path bundled. However, as noted before, any edge that has another edge with the same endpoints cannot be a $(u,v)$-cutset, so it follows that $N^{H}_{m-1} < n-2$ unless $H$ consists of a $u$-$v$ path with exactly on edge bundled, that is, $H$ is a graph of the same form as $G$. 
\end{proof}

With this lemma, we can now begin to look at $(n,m)$-graphs, where $m \geq n$, to see when an optimal graph exists (and when it does not). We can handle a wide variety of cases together. Throughout the rest of this chapter, $G$ will denote (any) one of the graphs $G_{n,m}$ in Proposition~\ref{onlypossibleoptimal}.

\begin{prop}\label{mainprop}
If $n\geq 4$ and $m>n$ there is no optimal $(n,m)$-graph.
\end{prop}
\begin{proof}
Let $m > n \geq 4$. 
%By Proposition~\ref{onlypossibleoptimal}, the only possible optimal graphs are ones whose underlying simple graph is a path, with exactly one of the edges bundled with the remaining edges; let $G$ denote any such graph (they all have the same split reliability). 
From the proof of Proposition~\ref{onlypossibleoptimal}, $G$ has 
\[ N_{n-2} = 1+(n-2)(m-(n-2)).\] 
Therefore, by  Lemma~\ref{mainobservation}, to show that no optimal $(n,m)$-graph exists, it suffices to present an $(n,m)$-graph $H$ for which $N^{H}_{n-2}$ is larger than $1+(n-2)(m-(n-2))$.

Let $H$ be a $(n,m)$-graph formed from $G$ by moving an edge from the bundle to another edge (see Figure~\ref{fig:n>=4}). The bundles of edges have cardinalities $2$ and $m-n+1$. Again, any such graph $H$ (with terminals $s$ and $t$ at the ends of the path) has the same split reliability. As $m > n$, $H$ is not isomorphic to $G$ (the latter has no edge bundle of size $2$ while the former does). 

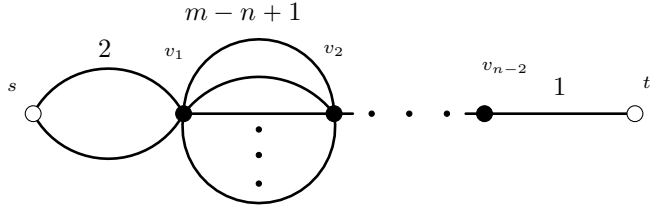
\begin{figure}[t]
\centering
\begin{tikzpicture}[line cap=round,line join=round,>=triangle 45,x=1cm,y=1cm]
\clip(-15,4.05) rectangle (-6,7);
\draw [line width=1pt] (-12.39,5.3)-- (-10.39,5.3);
\draw [line width=1pt] (-8.39,5.3)-- (-6.39,5.3);
\draw (-7.6,5.9) node[anchor=north west] {1};
\draw [shift={(-13.39,4.7666666666666675)},line width=1pt]  plot[domain=0.4899573262537275:2.6516353273360656,variable=\t]({1*1.1333333333333329*cos(\t r)+0*1.1333333333333329*sin(\t r)},{0*1.1333333333333329*cos(\t r)+1*1.1333333333333329*sin(\t r)});
\draw [shift={(-13.39,5.833333333333336)},line width=1pt]  plot[domain=3.6315499798435233:5.793227980925856,variable=\t]({1*1.1333333333333346*cos(\t r)+0*1.1333333333333346*sin(\t r)},{0*1.1333333333333346*cos(\t r)+1*1.1333333333333346*sin(\t r)});
%\draw [line width=2pt] (-14.39,5.3)-- (-12.39,5.3);
\draw (-13.65,6.4) node[anchor=north west] {$2$};
\draw [shift={(-11.39,5.289949494949494)},line width=1pt]  plot[domain=0.010050166661625747:3.1315424869281676,variable=\t]({1*1.000050505050505*cos(\t r)+0*1.000050505050505*sin(\t r)},{0*1.000050505050505*cos(\t r)+1*1.000050505050505*sin(\t r)});
\draw [shift={(-11.39,4.524591836734693)},line width=1pt]  plot[domain=0.659565020372448:2.482027633217345,variable=\t]({1*1.2654081632653065*cos(\t r)+0*1.2654081632653065*sin(\t r)},{0*1.2654081632653065*cos(\t r)+1*1.2654081632653065*sin(\t r)});
\draw [shift={(-11.39,5.125168067226892)},line width=1pt]  plot[domain=-3.3146752422769117:0.17308258868711818,variable=\t]({1*1.0151680672268888*cos(\t r)+0*1.0151680672268888*sin(\t r)},{0*1.0151680672268888*cos(\t r)+1*1.0151680672268888*sin(\t r)});
\draw (-12.5,6.9) node[anchor=north west] {$m-n+1$};
\draw [line width=1pt] (-10.39,5.3)-- (-10.14,5.3);
\draw [line width=1pt] (-8.64,5.3)-- (-8.39,5.3);
\begin{scriptsize}
\draw [fill=white] (-14.39,5.3) circle (3pt);
\draw[color=black] (-14.67,5.66) node {$s$};
\draw [fill=black] (-12.39,5.3) circle (3pt);
\draw[color=black] (-12.51,6.1) node {$v_{1}$};
\draw [fill=black] (-10.39,5.3) circle (3pt);
\draw[color=black] (-10.39,6.14) node {$v_{2}$};
\draw [fill=black] (-9.89,5.3) circle (1pt);
\draw [fill=black] (-9.39,5.3) circle (1pt);
\draw [fill=black] (-8.89,5.3) circle (1pt);
\draw [fill=black] (-8.39,5.3) circle (3pt);
\draw[color=black] (-8.12,5.9) node {$v_{n-2}$};
\draw [fill=white] (-6.39,5.3) circle (3pt);
\draw[color=black] (-6.23,5.72) node {$t$};
\draw [fill=black] (-11.39,5.09) circle (1pt);
\draw [fill=black] (-11.39,4.75) circle (1pt);
\draw [fill=black] (-11.39,4.37) circle (1pt);
\end{scriptsize}
\end{tikzpicture}
\caption{An $(n,m)$-graph $m > n \geq 4$ which has greater split reliability than $G$ when $p$ is close to $0$.}
\label{fig:n>=4}
\end{figure}

Let's compute $N^{H}_{n-2}$, by considering which edge/bundle is down (exactly one must be down for the spanning subgraph to fall into two components, one containing $s$, the other containing $t$).
It is not hard to see that
\[ N_{n-2}^{H}=2(m-n+1)(n-3)+2+(m-n+1).\]
Now some simplification shows that 
\[ N_{n-2}^{H} = 2(m-n+1)(n-3)+2+(m-n+1) > 1+(n-2)(m-(n-2)) = N_{n-2}^{G}  \] if and only if
\[ (m-n)(n-3)> 0,\]
which is true as $n\geq 4$ and $m > n$.
It follows that $H$ (with terminals $s$ and $t$) has larger split reliability than $G$ near $0$, so we conclude that no optimal $(n,m)$-graph exists in this case.
\end{proof}

\begin{prop}\label{propn=m}
If $n=m\geq5$ there is no optimal $(n,m)$-graph.
\end{prop}
\begin{proof}
Our approach again is to use Lemma~\ref{mainobservation} and Proposition~\ref{onlypossibleoptimal} to find an $(n,n)$-graph $H$ such that such $N^{H}_{n-2} > 1+ (n-2)(m-(n-2))$, which will show that no optimal $(n,n)$-graph exists.

For $n=m=5$, consider the $(5,5)$-graph $H$ as shown in Figure~\ref{fig:X and Y}.
\begin{figure}
\centering
\begin{tikzpicture}[line cap=round,line join=round,>=triangle 45,x=1cm,y=1cm]
\clip(-14.3,4.5) rectangle (-9.2,7.33);
\draw [line width=1pt] (-14.2,5.75)-- (-12.43,6.65);
\draw [line width=1pt] (-12.43,6.65)-- (-10,6.67);
\draw [line width=1pt] (-12.43,6.65)-- (-12.45,4.85);
\draw [line width=1pt] (-14.2,5.75)-- (-12.45,4.85);
\draw [line width=1pt] (-12.45,4.85)-- (-10,4.85);
\begin{scriptsize}
\draw [fill=black] (-14.2,5.75) circle (3pt);
\draw[color=black] (-14.03,6.18) node {$a$};
\draw [fill=black] (-12.43,6.65) circle (3pt);
\draw[color=black] (-12.27,7.08) node {$b$};
\draw [fill=black] (-12.45,4.85) circle (3pt);
\draw[color=black] (-12.29,5.28) node {$c$};
\draw [fill=white] (-10,4.85) circle (3pt);
\draw[color=black] (-9.83,5.28) node {$t$};
\draw [fill=white] (-10,6.67) circle (3pt);
\draw[color=black] (-9.83,7.1) node {$s$};
\end{scriptsize}
\end{tikzpicture}
\begin{tikzpicture}[line cap=round,line join=round,>=triangle 45,x=1cm,y=1cm]
\clip(-17.51,4.7) rectangle (-10.5,7.35);
\draw [line width=1pt] (-15,5.89)-- (-14,6.81);
\draw [line width=1pt] (-14,6.81)-- (-13,5.89);
\draw [line width=1pt] (-13,5.89)-- (-11,5.89);
\draw [line width=1pt] (-15,5.89)-- (-14,5.09);
\draw [line width=1pt] (-14,5.09)-- (-13,5.89);
\draw [line width=1pt] (-17,5.89)-- (-15,5.89);
\begin{scriptsize}
\draw [fill=black] (-15,5.89) circle (3pt);
\draw[color=black] (-15.19,6.3) node {$a$};
\draw [fill=black] (-14,6.81) circle (3pt);
\draw[color=black] (-13.99,7.2) node {$b$};
\draw [fill=black] (-13,5.89) circle (3pt);
\draw[color=black] (-12.83,6.32) node {$c$};
\draw [fill=white] (-11,5.89) circle (3pt);
\draw[color=black] (-10.83,6.32) node {$t$};
\draw [fill=black] (-14,5.09) circle (3pt);
\draw[color=black] (-13.83,4.82) node {$d$};
\draw [fill=white] (-17,5.89) circle (3pt);
\draw[color=black] (-16.83,6.32) node {$s$};
\end{scriptsize}
\end{tikzpicture}
\caption{A $(5,5)$-graph and a $(6,6)$-graph that have higher split reliability than $G$ when $p$ is close to $0$.}
\label{fig:X and Y}
\end{figure}
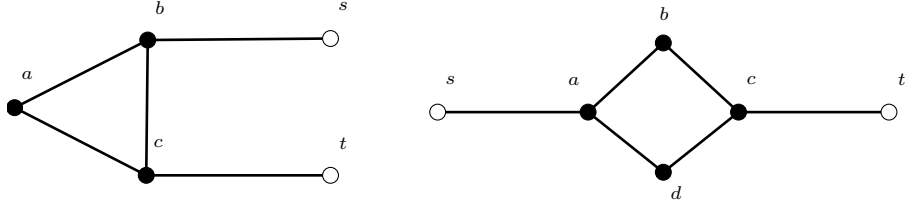
%\begin{figure}
%\centering
%\begin{tikzpicture}[line cap=round,line join=round,>=triangle 45,x=1cm,y=1cm]
%\clip(-14.3,4.5) rectangle (-9.2,7.33);
%\draw [line width=1pt] (-14.2,5.75)-- (-12.43,6.65);
%\draw [line width=1pt] (-12.43,6.65)-- (-10,6.67);
%\draw [line width=1pt] (-12.43,6.65)-- (-12.45,4.85);
%\draw [line width=1pt] (-14.2,5.75)-- (-12.45,4.85);
%\draw [line width=1pt] (-12.45,4.85)-- (-10,4.85);
%\begin{scriptsize}
%\draw [fill=black] (-14.2,5.75) circle (3pt);
%\draw[color=black] (-14.03,6.18) node {$a$};
%\draw [fill=black] (-12.43,6.65) circle (3pt);
%\draw[color=black] (-12.27,7.08) node {$b$};
%\draw [fill=black] (-12.45,4.85) circle (3pt);
%\draw[color=black] (-12.29,5.28) node {$c$};
%\draw [fill=white] (-10,4.85) circle (3pt);
%\draw[color=black] (-9.83,5.28) node {$t$};
%\draw [fill=white] (-10,6.67) circle (3pt);
%\draw[color=black] (-9.83,7.1) node {$s$};
%\end{scriptsize}
%\end{tikzpicture}
%    \caption{A $(5,5)$-graph which has higher split reliability than $G$ when $p$ is close to $0$.}
%\label{fig:X}
%\end{figure}
We find that $N_{n-2}^{H}=N_{3}^{H}=8 > 7 = 1+ (n-2)(m-(n-2))$, so there is no optimal $(n,m)$-graph when $n=m=5$.
Likewise, for $n=m=6$, consider the $(6,6)$-graph $H$ in Figure \ref{fig:X and Y}.
%\begin{figure}
%    \centering
%\begin{tikzpicture}[line cap=round,line join=round,>=triangle 45,x=1cm,y=1cm]
%\clip(-17.51,4.7) rectangle (-10.5,7.35);
%\draw [line width=1pt] (-15,5.89)-- (-14,6.81);
%\draw [line width=1pt] (-14,6.81)-- (-13,5.89);
%\draw [line width=1pt] (-13,5.89)-- (-11,5.89);
%\draw [line width=1pt] (-15,5.89)-- (-14,5.09);
%\draw [line width=1pt] (-14,5.09)-- (-13,5.89);
%\draw [line width=1pt] (-17,5.89)-- (-15,5.89);
%\begin{scriptsize}
%\draw [fill=black] (-15,5.89) circle (3pt);
%\draw[color=black] (-15.19,6.3) node {$a$};
%\draw [fill=black] (-14,6.81) circle (3pt);
%\draw[color=black] (-13.99,7.2) node {$b$};
%\draw [fill=black] (-13,5.89) circle (3pt);
%\draw[color=black] (-12.83,6.32) node {$c$};
%\draw [fill=white] (-11,5.89) circle (3pt);
%\draw[color=black] (-10.83,6.32) node {$t$};
%\draw [fill=black] (-14,5.09) circle (3pt);
%\draw[color=black] (-13.83,4.82) node {$d$};
%\draw [fill=white] (-17,5.89) circle (3pt);
%\draw[color=black] (-16.83,6.32) node {$s$};
%\end{scriptsize}
%\end{tikzpicture}
%    \caption{A $(6,6)$-graph which has higher split reliability than $G$ when $p$ is close to $0$.}
%    \label{fig:Y}
%\end{figure}
We find that $N_{n-2}^{H}=12 > 1+(n-2)(m-(n-2)) = 9$, so there is no optimal $(n,m)$-graph when $n=m=6$.

Now assume $n = m \geq 7$. We will construct our graph $H$ as follows. Begin with a path of length $n-3$ with $s$ and $t$ as the endpoints. With the remaining 3 edges and 2 vertices, construct a triangle with $t$ as one of the points (for an example see Figure \ref{fig:Z}).
%\begin{figure}
%\centering
%\begin{tikzpicture}[line cap=round,line join=round,>=triangle 45,x=1cm,y=1cm]
%\clip(-17.5,4.75) rectangle (-7.5,7.5);
%\draw [line width=2pt] (-11,5.89)-- (-9,5.89);
%\draw [line width=2pt] (-17,5.89)-- (-15,5.89);
%\draw [line width=2pt] (-15,5.89)-- (-13,5.89);
%\draw [line width=2pt] (-13,5.89)-- (-11,5.89);
%\draw [line width=2pt] (-9,5.89)-- (-8,6.77);
%\draw [line width=2pt] (-8,6.77)-- (-8,5.01);
%\draw [line width=2pt] (-8,5.01)-- (-9,5.89);
%\begin{scriptsize}
%\draw [fill=black] (-15,5.89) circle (2.5pt);
%\draw[color=black] (-15.19,6.3) node {$a$};
%\draw [fill=black] (-11,5.89) circle (2.5pt);
%\draw[color=black] (-10.83,6.32) node {$c$};
%\draw [fill=black] (-9,5.89) circle (2.5pt);
%\draw[color=black] (-9.07,6.34) node {$t$};
%\draw [fill=black] (-17,5.89) circle (2.5pt);
%\draw[color=black] (-16.83,6.32) node {$s$};
%\draw [fill=black] (-13,5.89) circle (2.5pt);
%\draw[color=black] (-12.83,6.32) node {$b$};
%\draw [fill=black] (-8,6.77) circle (2.5pt);
%\draw[color=black] (-7.83,7.2) node {$d$};
%\draw [fill=black] (-8,5.01) circle (2.5pt);
%\draw[color=black] (-7.83,5.44) node {$e$};
%\end{scriptsize}
%\end{tikzpicture}
%\caption{Graph $H$ when $n=m=7$}
%\label{fig:Z}
%\end{figure}
\begin{figure}
\centering
\begin{tikzpicture}[line cap=round,line join=round,>=triangle 45,x=1cm,y=1cm]
\clip(-17.5,4.75) rectangle (-7.5,7.5);
\draw [line width=1pt] (-17,5.89)-- (-15,5.89);
\draw [line width=1pt] (-12.75,5.89)-- (-10.75,5.89);
\draw [line width=1pt] (-10.75,5.89)-- (-8.75,6.89);
\draw [line width=1pt] (-8.75,6.89)-- (-8.75,4.89);
\draw [line width=1pt] (-8.75,4.89)-- (-10.75,5.89);
\draw [line width=1pt] (-15,5.89)-- (-14.75,5.89);
\draw [line width=1pt] (-13,5.89)-- (-12.75,5.89);
\begin{scriptsize}
\draw [fill=white] (-17,5.89) circle (3pt);
\draw[color=black] (-17.167465358602417,6.378557562650185) node {$s$};
\draw [fill=black] (-15,5.89) circle (3pt);
\draw[color=black] (-15.196521006221204,6.642952536750104) node {$v_{1}$};
\draw [fill=black] (-14.444,5.89) circle (1pt);
\draw [fill=black] (-13.888,5.89) circle (1pt);
\draw [fill=black] (-13.333,5.89) circle (1pt);
\draw [fill=black] (-12.75,5.89) circle (3pt);
\draw[color=black] (-12.901091912899178,6.666988443486461) node {$v_{n-4}$};
\draw [fill=white] (-10.75,5.89) circle (3pt);
\draw[color=black] (-10.918129607149785,6.450665282859253) node {$t$};
\draw [fill=black] (-8.75,6.89) circle (3pt);
\draw[color=black] (-8.16601828583699,7.147706578213587) node {$v_{n-3}$};
\draw [fill=black] (-8.75,4.89) circle (3pt);
\draw[color=black] (-8.16601828583699,5.2007981325687265) node {$v_{n-2}$};
\end{scriptsize}
\end{tikzpicture}
\caption{An $(n,n)$-graph ($n \geq 7$) which has higher split reliability than $G$ when $p$ is close to $0$.}
\label{fig:Z}
\end{figure}
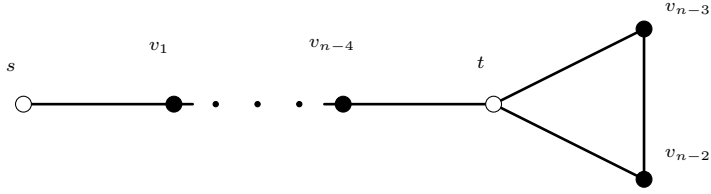
We can determine that $N_{n-2}^{H}=3(n-3) > 1+(n-2)(m-(n-2)) = 1+2(n-2)$ as $n \geq 7$. Again it follows that there is no optimal $(n,m)$-graph when $n = m \geq 7$.
\end{proof}

In the remaining cases, we shall see that optimal graphs {\bf always} exist.

\begin{prop}\label{prop3}
For any $m \geq 2$, there is an optimal $(3,m)$-graph.
\end{prop}
\begin{proof}
We have 3 possible families of graphs, based on the underlying simple graph (see Figure~\ref{fig:B,C,D}):
\begin{itemize}
\item a path of length $2$ with $s$ and $t$ at the endpoints of the path (graph $B$),
\item a path of length $2$ with one terminal (say $s$) at the endpoint of the path and the other terminal ($t$) in the centre of the path (graph  $C$), or
\item a triangle (graph $D$).
\end{itemize} 
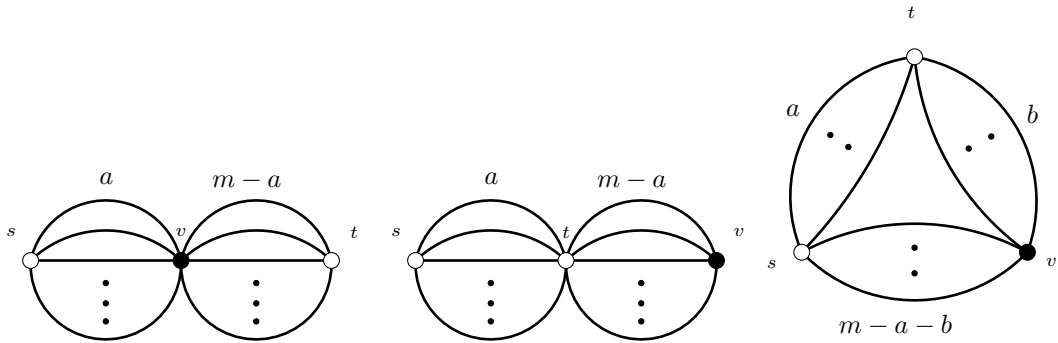
\begin{figure}[h]
\centering
\begin{tikzpicture}[line cap=round,line join=round,>=triangle 45,x=1cm,y=1cm]
\clip(-15,4.2) rectangle (-10,6.7);
\draw [shift={(-13.39,5.075)},line width=1pt]  plot[domain=0.2213144423477918:2.9202782112420014,variable=\t]({1*1.025*cos(\t r)+0*1.025*sin(\t r)},{0*1.025*cos(\t r)+1*1.025*sin(\t r)});
\draw [shift={(-11.39,5.075)},line width=1pt]  plot[domain=0.2213144423477918:2.9202782112420014,variable=\t]({1*1.025*cos(\t r)+0*1.025*sin(\t r)},{0*1.025*cos(\t r)+1*1.025*sin(\t r)});
\draw [shift={(-13.39,4.25)},line width=1pt]  plot[domain=0.8097835725701651:2.331809081019628,variable=\t]({1*1.45*cos(\t r)+0*1.45*sin(\t r)},{0*1.45*cos(\t r)+1*1.45*sin(\t r)});
\draw [shift={(-13.39,5.2511904761904775)},line width=1pt]  plot[domain=-3.1903634719352287:0.048770818345435486,variable=\t]({1*1.0011904761904744*cos(\t r)+0*1.0011904761904744*sin(\t r)},{0*1.0011904761904744*cos(\t r)+1*1.0011904761904744*sin(\t r)});
\draw [shift={(-11.39,4.25)},line width=1pt]  plot[domain=0.8097835725701651:2.331809081019628,variable=\t]({1*1.45*cos(\t r)+0*1.45*sin(\t r)},{0*1.45*cos(\t r)+1*1.45*sin(\t r)});
\draw [shift={(-11.39,5.2511904761904775)},line width=1pt]  plot[domain=-3.1903634719352287:0.048770818345435486,variable=\t]({1*1.0011904761904744*cos(\t r)+0*1.0011904761904744*sin(\t r)},{0*1.0011904761904744*cos(\t r)+1*1.0011904761904744*sin(\t r)});
\draw [line width=1pt] (-14.39,5.3)-- (-12.39,5.3);
\draw [line width=1pt] (-12.39,5.3)-- (-10.39,5.3);
\draw (-13.6,6.6) node[anchor=north west] {$a$};
\draw (-12.1,6.6) node[anchor=north west] {$m-a$};
\begin{scriptsize}
\draw [fill=white] (-14.39,5.3) circle (3pt);
\draw[color=black] (-14.65,5.68) node {$s$};
\draw [fill=white] (-10.39,5.3) circle (3pt);
\draw[color=black] (-10.09,5.68) node {$t$};
\draw [fill=black] (-12.39,5.3) circle (3pt);
\draw[color=black] (-12.39,5.68) node {$v$};
\draw [fill=black] (-13.39,5) circle (1pt);
\draw [fill=black] (-13.39,4.73) circle (1pt);
\draw [fill=black] (-13.39,4.49) circle (1pt);
\draw [fill=black] (-11.39,5) circle (1pt);
\draw [fill=black] (-11.39,4.73) circle (1pt);
\draw [fill=black] (-11.39,4.49) circle (1pt);
\end{scriptsize}
\end{tikzpicture}
%    \caption{Graph $B$}    \label{figB}
%\end{figure}
%\begin{figure}[h]
%    \centering
\begin{tikzpicture}[line cap=round,line join=round,>=triangle 45,x=1cm,y=1cm]
\clip(-15,4.2) rectangle (-10,7);
\draw [shift={(-13.39,5.075)},line width=1pt]  plot[domain=0.2213144423477918:2.9202782112420014,variable=\t]({1*1.025*cos(\t r)+0*1.025*sin(\t r)},{0*1.025*cos(\t r)+1*1.025*sin(\t r)});
\draw [shift={(-11.39,5.075)},line width=1pt]  plot[domain=0.2213144423477918:2.9202782112420014,variable=\t]({1*1.025*cos(\t r)+0*1.025*sin(\t r)},{0*1.025*cos(\t r)+1*1.025*sin(\t r)});
\draw [shift={(-13.39,4.25)},line width=1pt]  plot[domain=0.8097835725701651:2.331809081019628,variable=\t]({1*1.45*cos(\t r)+0*1.45*sin(\t r)},{0*1.45*cos(\t r)+1*1.45*sin(\t r)});
\draw [shift={(-13.39,5.2511904761904775)},line width=1pt]  plot[domain=-3.1903634719352287:0.048770818345435486,variable=\t]({1*1.0011904761904744*cos(\t r)+0*1.0011904761904744*sin(\t r)},{0*1.0011904761904744*cos(\t r)+1*1.0011904761904744*sin(\t r)});
\draw [shift={(-11.39,4.25)},line width=1pt]  plot[domain=0.8097835725701651:2.331809081019628,variable=\t]({1*1.45*cos(\t r)+0*1.45*sin(\t r)},{0*1.45*cos(\t r)+1*1.45*sin(\t r)});
\draw [shift={(-11.39,5.2511904761904775)},line width=1pt]  plot[domain=-3.1903634719352287:0.048770818345435486,variable=\t]({1*1.0011904761904744*cos(\t r)+0*1.0011904761904744*sin(\t r)},{0*1.0011904761904744*cos(\t r)+1*1.0011904761904744*sin(\t r)});
\draw [line width=1pt] (-14.39,5.3)-- (-12.39,5.3);
\draw [line width=1pt] (-12.39,5.3)-- (-10.39,5.3);
\draw (-13.6,6.6) node[anchor=north west] {$a$};
\draw (-12.1,6.6) node[anchor=north west] {$m-a$};
\begin{scriptsize}
\draw [fill=white] (-14.39,5.3) circle (3pt);
\draw[color=black] (-14.65,5.68) node {$s$};
\draw [fill=black] (-10.39,5.3) circle (3pt);
\draw[color=black] (-10.09,5.68) node {$v$};
\draw [fill=white] (-12.39,5.3) circle (3pt);
\draw[color=black] (-12.39,5.68) node {$t$};
\draw [fill=black] (-13.39,5) circle (1pt);
\draw [fill=black] (-13.39,4.73) circle (1pt);
\draw [fill=black] (-13.39,4.49) circle (1pt);
\draw [fill=black] (-11.39,5) circle (1pt);
\draw [fill=black] (-11.39,4.73) circle (1pt);
\draw [fill=black] (-11.39,4.49) circle (1pt);
\end{scriptsize}
\end{tikzpicture}
%    \caption{Graph $C$}    \label{figC}
%\end{figure}
%\begin{figure}[h]    
%    \centering
\begin{tikzpicture}[line cap=round,line join=round,>=triangle 45,x=1cm,y=1cm]
\clip(-15,5) rectangle (-10.5,9.9);
\draw [shift={(-13.197916666666671,6.888124999999998)},line width=1pt]  plot[domain=-0.3552437891876963:1.4018004334516467,variable=\t]({1*1.949649418716809*cos(\t r)+0*1.949649418716809*sin(\t r)},{0*1.949649418716809*cos(\t r)+1*1.949649418716809*sin(\t r)});
\draw [shift={(-9.279141104294483,9.148957055214728)},line width=1pt]  plot[domain=3.23570820534283:4.09403374610071,variable=\t]({1*3.6068212451059054*cos(\t r)+0*3.6068212451059054*sin(\t r)},{0*3.6068212451059054*cos(\t r)+1*3.6068212451059054*sin(\t r)});
\draw [shift={(-12.87,3.439473684210527)},line width=1pt]  plot[domain=1.0745694286015435:2.06702322498825,variable=\t]({1*3.1505263157894734*cos(\t r)+0*3.1505263157894734*sin(\t r)},{0*3.1505263157894734*cos(\t r)+1*3.1505263157894734*sin(\t r)});
\draw [shift={(-12.87,7.647812499999999)},line width=1pt]  plot[domain=3.905826016947017:5.518951943822362,variable=\t]({1*2.077812499999999*cos(\t r)+0*2.077812499999999*sin(\t r)},{0*2.077812499999999*cos(\t r)+1*2.077812499999999*sin(\t r)});
\draw [shift={(-12.650407079646017,6.950619469026548)},line width=1pt]  plot[domain=1.6883518477764776:3.548276815139157,variable=\t]({1*1.872302595637978*cos(\t r)+0*1.872302595637978*sin(\t r)},{0*1.872302595637978*cos(\t r)+1*1.872302595637978*sin(\t r)});
\draw [shift={(-18.817184115523478,10.508375451263545)},line width=1pt]  plot[domain=5.514800841362223:6.005013128732998,variable=\t]({1*6.184939617925903*cos(\t r)+0*6.184939617925903*sin(\t r)},{0*6.184939617925903*cos(\t r)+1*6.184939617925903*sin(\t r)});
%\draw [line width=2pt] (-12.87,8.81)-- (-14.37,6.21);
%\draw [line width=2pt] (-14.37,6.21)-- (-11.37,6.21);
%\draw [line width=2pt] (-11.37,6.21)-- (-12.87,8.81);
\draw (-14.7,8.3) node[anchor=north west] {$a$};
\draw (-11.5,8.3) node[anchor=north west] {$b$};
\draw (-14,5.5) node[anchor=north west] {$m-a-b$};
\begin{scriptsize}
\draw [fill=white] (-12.87,8.81) circle (3pt);
\draw[color=black] (-12.91,9.4) node {$t$};
\draw [fill=white] (-14.37,6.21) circle (3pt);
\draw[color=black] (-14.77,6.06) node {$s$};
\draw [fill=black] (-11.37,6.21) circle (3pt);
\draw[color=black] (-11.05,6.08) node {$v$};
\draw [fill=black] (-13.99,7.77) circle (1pt);
\draw [fill=black] (-13.75,7.61) circle (1pt);
\draw [fill=black] (-12.15,7.59) circle (1pt);
\draw [fill=black] (-11.85,7.75) circle (1pt);
\draw [fill=black] (-12.87,6.27) circle (1pt);
\draw [fill=black] (-12.87,5.93) circle (1pt);
\end{scriptsize}
\end{tikzpicture}
\caption{Graphs $B$, $C$, and $D$ (the edge weights indicate the size of the edge bundles).}
\label{fig:B,C,D}
\end{figure}

Beginning with the graph $B$, without loss of generality assume $1\leq a\leq m-a < m$. We see that
\begin{eqnarray*}
\splitRelp{B}{s}{t}&=& (1-(1-p)^{a})(1-p)^{m-a}+(1-(1-p)^{m-a})(1-p)^{a}\\
&=& (1-p)^{a}+(1-p)^{m-a}-2(1-p)^{m}.
\end{eqnarray*}
We want to show that this is largest when $a=1$, which yields our optimal graph $G$ from Proposition~\ref{onlypossibleoptimal}. The last term, $2(1-p)^{m}$, is fixed, so consider $f_{a,m}=(1-p)^{a}+(1-p)^{m-a}$. For $p\in(0,1)$ and $2\leq a\leq m-a$:
\begin{eqnarray*}
f_{a-1,m}-f_{a,m} &=& (1-p)^{a-1}+(1-p)^{m-(a-1)}-(1-p)^{a}-(1-p)^{m-a}\\
&=&p((1-p)^{a-1}-(1-p)^{m-a})>0.
\end{eqnarray*}
So $f_{a,m}$, and hence $\splitRelp{B}{s}{t}$, is largest over all choices of $a$ and all $p\in(0,1)$ when $a$ is as small as possible, that is, when $a=1$ (which gives us one of the graphs $G$), so among all graphs $B$, the graph $G$ has uniformly for all $p \in (0,1)$ the largest split reliability.

We now compare this to the split reliability of graphs of the other types ($C$ and $D$). For a graph of form $C$ (see the second graph in Figure~\ref{fig:B,C,D}). Again, assume $1\leq a,m-a \leq m$:
\begin{eqnarray*}
\splitRelp{C}{s}{t}&=&(1-p)^{a}(1-(1-p)^{m-a})\\
&=& (1-p)^{a}-(1-p)^{m}
\end{eqnarray*}
This polynomial is uniformly largest (for all $p \in (0,1)$) when $a = 1$, but then 
\begin{eqnarray*}
& &\splitRelp{G}{s}{t}-\splitRelp{C}{s}{t}\\
&=& (1-p)+(1-p)^{m-1}-2(1-p)^{m}-((1-p)-(1-p)^{m})\\
&=& (1-p)^{m-1}-(1-p)^{m}>0
\end{eqnarray*}
Thus the split reliability of $G$ is uniformly greater than any such graph $C$. 

Finally, we have to consider $D$ (see the third graph in Figure~\ref{fig:B,C,D}), where \linebreak $1 \leq a,b,m-a-b \leq m$. An easy calculation shows that
\begin{eqnarray*}
\splitRelp{D}{s}{t}&=& (1-p)^{a}[(1-p)^{b}+(1-p)^{m-a-b}-2(1-p)^{m-a}]
\end{eqnarray*}
Looking at the expression in square brackets, we can see it is of the same form as the polynomial we calculated for the split reliability of graphs $B$. We know that this polynomial is uniformly largest when $b=1$, so
\begin{eqnarray*}
\splitRelp{D}{s}{t}&\leq& (1-p)^{a}[(1-p)+(1-p)^{m-a-1}-2(1-p)^{m-a}]\\
&=& (1-p)^{a+1}+(1-p)^{m-1}-2(1-p)^{m}\\
&<& (1-p)+(1-p)^{m-1}-2(1-p)^{m}\\
&=&\splitRelp{G}{s}{t}.
\end{eqnarray*}

It follows that $G$ is an optimal $(3,m)$-graph (and in fact, the graphs of this form are the only optimal such graphs).
\end{proof}

We have just one case left. 

\begin{prop}\label{prop4}
There is an optimal $(4,4)$-graph.
\end{prop}
\begin{proof}

Let $H$ be any $(4,4)$-graph with terminals $u$ and $v$; then we can write 
\begin{eqnarray*}
\splitRelp{H}{u}{v}& = & N^H_{2}p^2(1-p)^2+N^H_{3}p^3(1-p),
\end{eqnarray*}
where $N^H_{2}  \geq 1$ and $N^H_{3} \geq 0$. For a graph $G$ with $n = m = 4$ with terminals $s$ and $t$ as in Proposition~\ref{onlypossibleoptimal}, $N_2 = 5$ and $N_3 = 2$. From the proof of the proposition, $N^H_3 < 2$ unless $H$ is of the same form as $G$ (and hence has the same split reliability), or $H$ has a path of length $2$ between $u$ and $v$, and the vertex $w$ off this path is attached to a single vertex. In the latter case, $N^H_2 = 4 < 5$ and hence $\splitRelp{H}{u}{v} < \splitRelp{G}{s}{t}$ for all $p \in (0,1)$.

So in all other cases, $N^H_3 < 2$, so it suffices to show in these cases that $N^H_2 \leq 5$. Suppose, to reach a contradiction, that $N^H_2 > 5$. However, there are at most ${4 \choose 2} = 6$ subsets of cardinality $2$ of the edge set of $H$, so the only way that $N^H_2 > 5$ is if {\em every} set of two edges forms a spanning subgraph of $H$ with two components, with one containing $u$, the other containing $v$. It follows that $\{u,v\}$ is not an edge. Let $P$ be a shortest $u$-$v$ path. Clearly there cannot be 2 edges not belonging to $P$, so $P$ must be a path of length $3$, and hence $H$ is a graph formed from $P$ by adding an additional edge \underline{not} between two adjacent vertices (as $N^H_3 < 2$, $H$ is not of the same form as $G$) or between $u$ and $v$. However, then it is easy to check that indeed $N^H_2 = 5$. 

We conclude that there is an optimal $(4,4)$-graph (namely, any such graph $G$).  
\end{proof}

Putting together the opening remarks of the chapter, and Propositions~\ref{proptrees}, \ref{mainprop}, \ref{propn=m}, \ref{prop3} and \ref{prop4}, we derive our main result:  

\begin{theorem}
For $n \geq 2$ and $m \geq n-1$, there is an optimal $(n,m)$-graph if and only if $n \leq 3$, $m=n-1$, or $n = m = 4$. \qed
\end{theorem}

\section{Open Problems}

While we have determined the conditions on $n$ and $m$ for optimal graphs to exist for strongly connected reliability, the problem seems more difficult (and more subtle) if one restricts to simple graphs, that is graphs without multiple edges. Calculations on small graphs indicate that there is an optimal simple graph 
\begin{itemize}
    \item if $n\leq 5$, 
    \item when $n=6$ unless $m=6$ or $m=8$, and 
    \item when $n=7$ if and only if $m=6$ or $14\leq m\leq 21$.
\end{itemize} 
Furthermore, it is not hard to see that there is always an optimal simple graph when $m=n-1$, $m=\binom{n}{2}$, and $m=\binom{n}{2}-1$, where we have an optimal graph when we choose $s$ and $t$ to be the two non-adjacent vertices. 

When $n = m \geq 6$, we can show that there is no optimal simple graph as follows (we only sketch the proof here). Consider a graph $S=S_{n}$ that consists of a single path of length $n-2$ between vertices $s$ and $t$ with our one remaining vertex connected to two adjacent vertices on the path (see Figure \ref{bestpossible_n=m_simple}).
\begin{figure}[t]
    \centering
\begin{tikzpicture}[line cap=round,line join=round,>=triangle 45,x=1cm,y=1cm]
\clip(-15,3.9) rectangle (-6,5.9);
\draw [line width=1pt] (-14.39,5.3)-- (-12.39,5.3);
\draw [line width=1pt] (-12.39,5.3)-- (-10.39,5.3);
\draw [line width=1pt] (-8.39,5.3)-- (-6.39,5.3);
\draw [line width=1pt] (-11.39,4.2)-- (-12.39,5.3);
\draw [line width=1pt] (-11.39,4.2)-- (-10.39,5.3);
\draw [line width=1pt] (-10.39,5.3)-- (-10.14,5.3);
\draw [line width=1pt] (-8.64,5.3)-- (-8.39,5.3);
\begin{scriptsize}
\draw [fill=white] (-14.39,5.3) circle (3pt);
\draw[color=black] (-14.67,5.66) node {$s$};
\draw [fill=black] (-12.39,5.3) circle (3pt);
\draw[color=black] (-12.19,5.66) node {$v_{1}$};
\draw [fill=black] (-10.39,5.3) circle (3pt);
\draw[color=black] (-10.21,5.66) node {$v_{2}$};
\draw [fill=black] (-9.89,5.3) circle (1pt);
\draw [fill=black] (-9.39,5.3) circle (1pt);
\draw [fill=black] (-8.89,5.3) circle (1pt);
\draw [fill=black] (-8.39,5.3) circle (3pt);
\draw[color=black] (-8.12,5.66) node {$v_{n-3}$};
\draw [fill=white] (-6.39,5.3) circle (3pt);
\draw[color=black] (-6.23,5.72) node {$t$};
\draw [fill=black] (-11.39,4.2) circle (3pt);
\draw[color=black] (-10.8,4.2) node {$v_{n-2}$};
\end{scriptsize}
\end{tikzpicture}
    \caption{Example of $S_{n}$, the only candidate optimal $(n,n)$- simple graph.}
    \label{bestpossible_n=m_simple}
\end{figure}
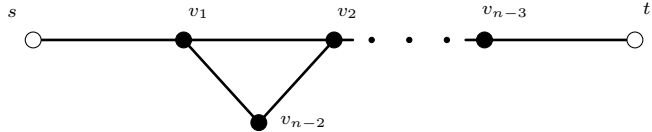
It can be shown that $S$ is the only possible optimal graph since it maximizes $N^{S}_{m-1}=n-3$ and has $N^{S}_{n-2}=3(n-3)+2$, which is the largest $N_{n-2}$ among those graphs with $N_{m-1}=n-3$.

We prove that there is no optimal graph when $n=m\geq 6$ by finding another $(n,n)$-graph with greater $N_{n-2}$. Consider the $(n,n)$-graph $R=R_{n}$ that consists of the cycle of size $n-2$ with two vertices on the cycle $u$ and $v$ such that the paths between $u$ and $v$ are of length $\lfloor\frac{n-2}{2} \rfloor$ and $\lceil\frac{n-2}{2} \rceil$, with $s$ and $t$ off of the cycle and adjacent to $u$ and $v$ respectively (see Figure \ref{bestpossible_n=m_simple_contradiction}).
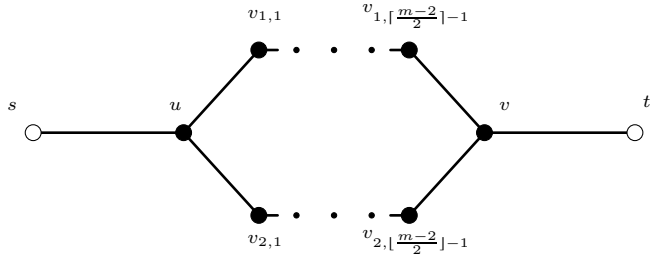
\begin{figure}
    \centering
\begin{tikzpicture}[line cap=round,line join=round,>=triangle 45,x=1cm,y=1cm]
\clip(-15,3.5) rectangle (-6,7);
\draw [line width=1pt] (-14.39,5.3)-- (-12.39,5.3);
\draw [line width=1pt] (-8.39,5.3)-- (-6.39,5.3);
\draw [line width=1pt] (-11.39,4.2)-- (-12.39,5.3);
\draw [line width=1pt] (-12.39,5.3)-- (-11.39,6.4);
\draw [line width=1pt] (-9.39,4.2)-- (-8.39,5.3);
\draw [line width=1pt] (-9.39,6.4)-- (-8.39,5.3);
\draw [line width=1pt] (-11.39,6.4)-- (-11.14,6.4);
\draw [line width=1pt] (-9.64,6.4)-- (-9.39,6.4);
\draw [line width=1pt] (-11.39,4.2)-- (-11.14,4.2);
\draw [line width=1pt] (-9.64,4.2)-- (-9.39,4.2);
\begin{scriptsize}
\draw [fill=white] (-14.39,5.3) circle (3pt);
\draw[color=black] (-14.67,5.66) node {$s$};
\draw [fill=black] (-12.39,5.3) circle (3pt);
\draw[color=black] (-12.5,5.66) node {$u$};
\draw [fill=black] (-8.39,5.3) circle (3pt);
\draw[color=black] (-8.12,5.66) node {$v$};
\draw [fill=white] (-6.39,5.3) circle (3pt);
\draw[color=black] (-6.23,5.72) node {$t$};
\draw [fill=black] (-11.39,4.2) circle (3pt);
\draw[color=black] (-11.3,3.85) node {$v_{2,1}$};
\draw [fill=black] (-10.89,4.2) circle (1pt);
\draw [fill=black] (-10.39,4.2) circle (1pt);
\draw [fill=black] (-9.89,4.2) circle (1pt);
\draw [fill=black] (-9.39,4.2) circle (3pt);
\draw[color=black] (-9.3,3.85) node {$v_{2,\lfloor\frac{m-2}{2} \rfloor-1}$};
\draw [fill=black] (-11.39,6.4) circle (3pt);
\draw[color=black] (-11.3,6.85) node {$v_{1,1}$};
\draw [fill=black] (-10.89,6.4) circle (1pt);
\draw [fill=black] (-10.39,6.4) circle (1pt);
\draw [fill=black] (-9.89,6.4) circle (1pt);
\draw [fill=black] (-9.39,6.4) circle (3pt);
\draw[color=black] (-9.3,6.85) node {$v_{1,\lceil\frac{m-2}{2} \rceil-1}$};
\end{scriptsize}
\end{tikzpicture}
    \caption{Example of $R_{n}$, an $(n,n)$-simple graph $n\geq 6$ which has higher split reliability than $S_{n}$ when $p$ is close to 0.}
    \label{bestpossible_n=m_simple_contradiction}
\end{figure}
We find that $N^{R}_{n-2}=2(n-2)+\lfloor\frac{n-2}{2} \rfloor\lceil\frac{n-2}{2} \rceil$, which is greater than $N^{S}_{n-2}=3(n-3)+2$ when $n\geq 6$, so there is not an optimal graph when $n=m\geq 6$. 

The existence or nonexistence of optimal simple graphs in other cases is left as an open problem worth exploring.

%This method could be a good place to start when further discussing optimal simple graphs for split reliability.

%\newpage
%\bibliographystyle{plain}
%\bibliography{simple}

%\begin{thebibliography}{99}

%\bibitem{concave} Jason I. Brown, Yakup Koc and Robert Kooij, Inflection Points for Network Reliability, Telecommunication Systems Journal \textbf{56} (2014), 79-84.

%\bibitem{colbook} C.J. Colbourn, The Combinatorics of Network Reliability, Oxford University Press, New York (1987).

%\bibitem{devore} J.L. Devore, Probability and Statistics for Engineering and the Sciences, Cengage Learning, Boston (2015).

%\bibitem{sharpp} J.S. Provan, M.O. Ball, The Complexity of Counting Cuts and of Computing the Probability that a Graph is Connected, SIAM Journal on Computing. \textbf{12} (1983) 777-788.

\section*{Acknowledgements}  %===================================
 
J. Brown acknowledges research support from the Natural Sciences and Engineering Research Council of Canada (NSERC), grant RGPIN 2018-05227.

\bibliographystyle{amsplain}

\end{document}